\documentclass[letterpaper, 10 pt, conference]{ieeeconf}  % Comment this line out if you need a4paper

\IEEEoverridecommandlockouts                              % This command is only needed if 
\usepackage{graphics} % for pdf, bitmapped graphics files
\usepackage{epsfig} % for postscript graphics files
\usepackage{times} % assumes new font selection scheme installed
\usepackage{amsmath} % assumes amsmath package installed
\usepackage{amssymb}  % assumes amsmath package installed
\usepackage{xcolor}
\usepackage{algorithmic}
\usepackage{algorithm}
\usepackage{hyperref}
\usepackage{comment}

\usepackage{amsthm}

\newtheorem{theorem}{\textbf{Theorem}}

\newtheorem*{probstat*}{\textbf{Problem Statement}}

\newtheorem{lemma}{\textbf{Lemma}}
\newtheorem{define}{\textbf{Definition}}

\newtheorem{remark}{\textbf{Remark}}
\newtheorem{assumption}{\textbf{Assumption}}

\title{\LARGE \bf
Feasibility Evaluation of Quadratic Programs for Constrained Control 
}

\author{Panagiotis Rousseas$^{1}$ and Dimitra Panagou$^{2}$% <-this % stops a space
\thanks{This work was supported by the National Science Foundation (NSF) under Award No. 1942907.}% <-this % stops a space
\thanks{$^{1}$Panagiotis Rousseas is with the Division of Decision and Control Systems, School of Electrical Engineering and Computer Science, KTH Royal Institute of Technology, Stockholm Sweden 
        {\tt\small rousseas@kth.se}}%
\thanks{$^{2}$Dimitra Panagou is with the Department of Robotics and Department of Aerospace Engineering, University of Michigan, MI USA 
        {\tt\small dpanagou@umich.edu}}%
}

\begin{document}

\maketitle
\thispagestyle{empty}
\pagestyle{empty}

%%%%%%%%%%%%%%%%%%%%%%%%%%%%%%%%%%%%%%%%%%%%%%%%%%%%%%%%%%%%%%%%%%%%%%%%%%%%%%%%
{ \begin{abstract}
This paper presents a computationally-efficient method for evaluating the feasibility of Quadratic Programs (QPs) for online constrained control. Based on the duality principle, we first show that the feasibility of a QP can be determined by the solution of a properly-defined Linear Program (LP). Our analysis yields a LP that can be solved more efficiently compared to the original QP problem, and more importantly, is simpler in form and can be solved more efficiently compared to existing methods that assess feasibility via LPs. The computational efficiency of the proposed method compared to existing methods for feasibility evaluation is demonstrated in comparative case studies as well as a feasible-constraint selection problem, indicating its promise for online feasibility evaluation of optimization-based controllers.
\end{abstract}}

%%%%%%%%%%%%%%%%%%%%%%%%%%%%%%%%%%%%%%%%%%%%%%%%%%%%%%%%%%%%%%%%%%%%%%%%%%%%%%%%
\section{INTRODUCTION}
%Control in the presence of constraints has been an active and popular topic of research, motivated by real-world applications subject to safety, actuation, and operational limitations. Such limitations are encoded and enforced as constraints in optimization problems. The existence of valid solutions relies on the \textbf{feasibility} of the underlying optimization problem, i.e., on the non emptiness of the intersection of the sets of constraints. Certain classes of constraints are treated as hard, i.e., must be satisfied at all times, e.g.,  safety and actuator constraints. Constraints that need not be satisfied at all times can be modeled as soft constraints and can therefore be disregarded in face of infeasibility. %Evaluating feasibility of a given constrained set presents itself as a key challenge in solving such problems. 
Constrained optimization and control have been extensively studied in the literature \cite{book_optim}, motivated in part by real-world problems that operate under limitations. In the context of constrained control synthesis in particular, optimization-based control formulations such as Model Predictive Control (MPC) \cite{Schwenzer2021} and Nonlinear MPC (NMPC) have been widely adopted, while more recently, Control Barrier Functions (CBFs) \cite{8796030} enable the formulation of constraints that act as safety filters in Quadratic Programs (QPs) that compute safe controllers. In all cases, feasibility of the underlying optimization problems is necessary for the derivation of control actions. These problems become particularly challenging in the presence of multiple constraints, since the latter might render the underlying problem infeasible.   

\par 
This paper considers the problem of assessing the  feasibility of QPs for online control synthesis, motivated in part by the wide range of constrained control techniques that fall into this optimization form. We are particularly interested in finding algorithms for feasibility assessment that are computationally efficient, since online control synthesis typically requires the ability to find a solution and close the loop fast enough (e.g., in the order of milliseconds for certain applications). { To this end, we adopt a duality principle viewpoint to characterize the feasibility of the QP in terms of the boundedness of the solution of the dual problem. More specifically, we show {  in Section \ref{sec:sec:pp-dp} that the boundedness of the solution of a properly-defined LP} is necessary and sufficient for the feasibility of the original QP. We then show in Section \ref{sec:disreg_cons} that feasibility of the original QP can be evaluated by solving a properly-defined LP that is simpler compared to the LPs formulated in existing methods that determine the feasibility of a set of linear constraints \cite{boyd2004convex, Nocedal2006-dw}. Numerical evaluations in Section \ref{sec:computational-comparison} indicate the  computational efficiency of our proposed method compared to existing approaches, which is crucial for online feasibility evaluation for control synthesis. %This result enables us to evaluate whether a given set of linear constraints are compatible, and therefore define feasible optimization problems.

%This work proposes a solution to evaluate the feasibility of constrained optimization in the context of Quadratic Programs (QPs), since the latter addresses a variety of control problems. Hence, the proposed method can be used as a tool for planning algorithms that require feasible solutions to QPs. 
\par 
{ We then consider and demonstrate how our feasibility assessment method can be applied to the problem of selecting a subset of compatible constraints when the original set of constraints is incompatible (and therefore the original problem is infeasible) and the constraints can be considered soft.} Finding the maximal subset of feasible constraints (called the maxFS problem) \cite{doi:10.1080/03155986.2019.1607715} is known to be NP-hard and is addressed with heuristic approaches \cite{782859,Chinneck2010-qg, Chinneck1996, 10644243} or predetermined priority specifications \cite{Dubois1996}. %propose a search solution based on possibility theory, where a tree structure is formed based on a user-specified constraint hierarchy; however, still the question of evaluating the feasibility of the problem remains open. 
%More recently, heuristics for compatible constraint selection based on Lagrange Multipliers (LMs) \cite{https://doi.org/10.1002/aic.690170452} were developed in \cite{10644243}. 
However, these methods do not evaluate the feasibility of the original problem, which is what our method addresses. A concept similar to feasibility assessment arises in chance-constrained optimization problems (CCPs) \cite{Campi2011}; specifically, quantifying the feasibility of the optimal solution w.r.t. constraint removal, along with algorithms on how to disregard the constraints, have been recently addressed in \cite{9750913, ROMAO2023110601}. However, in this context, feasibility refers to the probability that the stochastic constraint will be violated via finite samples of the random variables, and hence is not related to non-emptiness of the constrained convex set; in other words, the existence and uniqueness of the solution to the corresponding scenario approach-optimization problems is still assumed. To show the potential of applying our method to online feasibility assessment for constrained control synthesis, in Section \ref{sec:constr_sel} we demonstrate that the derived LP can be used for selecting compatible constraints online when the original problem is found to be infeasible. Finally,  we conclude our paper in Section \ref{sec:conclusions} with some thoughts for future work.}

\par 
%{  More specifically, the proposed method determines whether a QP is feasible by solving a LP,} which is proven to determine the QP's feasibility {  by dual optimization principles. The idea is that given a set of linear constraints, we prove { that the boundedness of the solution of a properly defined LP} is necessary and sufficient for the feasibility of a given QP.} Our solution is demonstrated to be computationally superior to solving the original QP, and is thus more suitable for integration with existing QP-based methods for control design { while assessing feasibility of multiple constraints}. 

%{ 
%Based on the theoretical analysis, we derive an algorithm that achieves faster feasibility evaluation compared to existing LP-based methods.} Therefore, the main contribution of this work is a fast feasibility evaluation method in the form of a novel LP, The above are beneficial to the control community for performing feasibility evaluation in the context of linearly-constrained control. 
%}

\section{PROBLEM FORMULATION}
Consider the following QP:
\begin{equation}\label{eq:opt_prob_qp}
    \begin{split}
        u^\star = &\underset{u\in\mathbb{R}^m}{\textrm{arg}\min}\left\{ u^\top H u + F^\top u\right\}
        \triangleq 
        \underset{u\in\mathbb{R}^m}{\textrm{arg}\min}\left\{ \omega (u) \right\}
        , \\
        &\textrm{s.t.:} \quad A^\top u \leq B,
    \end{split}
\end{equation}
where $\omega:\mathbb{R}^m{ \rightarrow}\mathbb{R}$, $H\in\mathbb{R}^{m\times m}$ is a symmetric positive definite matrix, $F\in\mathbb{R}^m$, $A\in\mathbb{R}^{m\times C},B\in\mathbb{R}^{C}$, with $C\in\mathbb{N}$ denoting the number of constraints. We further denote the constraint index set as $\mathcal{C} = \left\{ 1,\cdots,C\right\}$, which is assumed to consist of two subsets, namely: the set of indexes $\mathcal{C}_h \subset \mathcal{C}$ that denote hard constraints, and the set of indexes $\mathcal{C}_s \subset \mathcal{C}$ that denote soft constraints, so that $\mathcal{C}_s \cup \mathcal{C}_h = \mathcal{C}$ and $\mathcal{C}_s \cap \mathcal{C}_h = \emptyset$. Hard constraints must be satisfied, while soft constraints should be satisfied as long as they do not cause Problem \eqref{eq:opt_prob_qp} to become infeasible (e.g., they can be thought of as non-critical specifications). 
\begin{assumption}\label{assum:hard_cons}
    Problem \eqref{eq:opt_prob_qp} is feasible when only the hard constraints are considered:
    % \begin{equation}
    \(
    \mathcal{U}_h = 
    \left\{ u\in\mathbb{R}^m |
        A_h^\top u \leq B_h
    \right\}
    \neq \emptyset,
    \)
    %\underset{i\in\mathcal{C}_h}{\bigcup\mathcal{S}_i}\left(x_u(t;\bar{x})\right) \neq \emptyset.
    % \notag 
    % \end{equation}
    where \(A_h \triangleq \left[ A_{i_1}, A_{i_2},\cdots,A_{i_{C_h}} \right] \in\mathbb{R}^{m\times C_h} \) and \(B_h \triangleq \left[ B_{i_1}, B_{i_2},\cdots,B_{i_{C_h}} \right]^\top \in\mathbb{R}^{ C_h} \) with \(\mathcal{C}_h = \left\{ i_1,i_2,\cdots,i_{C_h}\right\}\) denoting the index set of hard constraints, $C_h = \left| \mathcal{C}_h\right|$ is the cardinality of the set and $A_{i}\in\mathbb{R}^m,\ i\in\mathcal{C}$, denotes the $i$-th column of the matrix $A\in\mathbb{R}^{m\times C}$ in \eqref{eq:opt_prob_qp}.
\end{assumption}

\begin{probstat*}
Given Assumption \ref{assum:hard_cons}, find a method for evaluating the feasibility of {  Problem \eqref{eq:opt_prob_qp}}, { where in the latter some/all of the constraints may be allowed to be violated.} 
%{ If the problem is deemed infeasible, the second objective is to propose a method to disregard or relax a subset of the soft constraints so that the resulting problem is feasible.} 
\end{probstat*}
Our approach is to (i) introduce a method based on the duality principle for assessing the feasibility of Problem \eqref{eq:opt_prob_qp}, and (ii) if Problem \eqref{eq:opt_prob_qp} is infeasible, { formalize QPs that incorporate the violation of soft constraints, through the notion of a constraint configuration (which will be defined in the sequel) and which is introduced to explicitly account for} { such constraints that may be violated.}

\begin{remark}\label{rem:cbf}
    Problem \eqref{eq:opt_prob_qp} covers a wide range of constrained control synthesis problems; e.g., convergence and safety specifications can be encoded via { Control Lyapunov Functions (CLFs)} \cite{clf} and CBFs \cite{8796030} respectively, which lead to QP formulations. %Given a reference input, the associated controllers are termed CLF-QP and CBF-QP respectively. 
    Furthermore, linearly-constrained MPC problems also assume a QP formulation. 
    \end{remark}

\begin{remark}\label{rem:results}
    An illustrative example corresponding to a CBF-QP controller with time-varying constraints will be presented in the results { (see Sec. \ref{sec:constr_sel})}. %including time-varying constraint matrices as outlined in Rem. \ref{rem:cbf}.
\end{remark}

\section{METHODOLOGY}\label{sec:methodology}
\subsection{ A duality approach for assessing feasibility of \eqref{eq:opt_prob_qp}}\label{sec:sec:pp-dp}
We reformulate Problem \eqref{eq:opt_prob_qp}, { termed Primal Problem (PP),} using Lagrange multipliers (LMs) { $\lambda\in\mathbb{R}^C$} and constructing the augmented objective function ${ \bar{d}}:\mathbb{R}^m\times\mathbb{R}^C{ \rightarrow} \mathbb{R}$ as:
\begin{equation}\label{eq:opt_prob_kkt}
\begin{matrix}
\textrm{\textbf{PP:}} &
    \begin{split}
        u^\star = &\ \underset{u\in\mathbb{R}^m}{\textrm{arg}\min}\ 
                \underset{\lambda\geq 0}{\max}
                \left\{
                    \omega(u)
                    + \left( 
                        A^\top u -  B
                    \right)^\top\lambda
                \right\}
                \\
                \triangleq &\ 
                \underset{u\in\mathbb{R}^m}{\textrm{arg}\min}\ 
                \underset{\lambda\geq 0}{\max}
                \left\{
                    { \bar{d}}(u,\lambda)
                \right\}.
    \end{split}
\end{matrix}
\end{equation}
The corresponding Dual Problem (DP) \cite{book_optim} is formulated as:
\begin{equation}\label{eq:opt_prob_dual}
    \begin{matrix}
     \textrm{\textbf{DP:}} &
        \lambda^\star & = \underset{\lambda\geq 0 }{\textrm{arg}\max}\ 
                \underset{u\in\mathbb{R}^m}{\min}
                \left\{
                    { \bar{d}}(u,\lambda)
                \right\}.
    \end{matrix}
\end{equation}
The Karush-Kuhn-Tucker (KKT) condition for \eqref{eq:opt_prob_dual} is:
\begin{equation}\label{eq:stationary_kkt}
\begin{split}
        \left.\frac{\partial d}{\partial u}\right|_{\left(u^\star,\lambda\right)} &= 0 
        \overset{\eqref{eq:opt_prob_kkt}}{\Rightarrow} 
        2Hu^\star + F + A\lambda = 0 
        \Rightarrow \\
        \quad\quad \ \mathbb{R}^m \ni u^\star &= - \tfrac{1}{2}H^{-1}\left( F + A\lambda\right),
\end{split}
\end{equation}
while substituting \eqref{eq:stationary_kkt} in ${ \bar{d}}$ after some algebraic calculations yields { the dual cost function $d:\mathbb{R}^C\rightarrow\mathbb{R}$}:
\begin{equation}
% \(
\begin{split}
    &d(\lambda) \triangleq { \bar{d}}\left(u^\star(\lambda),\lambda\right) = \\
    &-\tfrac{1}{4}\|\sqrt{H^{-1}}A\lambda \|^2 - 
    \left(
        \tfrac{1}{2}F^\top H^{-1} A + B^\top
    \right) \lambda
    - \tfrac{1}{4}F^\top H^{-1}F, \notag 
    % \)
\end{split}
\end{equation}
with $\sqrt{\cdot}$ denoting the square root of a positive definite matrix (note that $H^{-1}$ is also positive definite).
Therefore, omitting the constant term, the DP \eqref{eq:opt_prob_dual} becomes:
\begin{equation}\label{eq:opt_prob_dual*}
\begin{split}
    \lambda^\star = \underset{\lambda\geq 0 }{\textrm{arg}\max}
                \left\{
                    -\tfrac{1}{4}\|\sqrt{H^{-1}}A\lambda \|^2 - 
    \left(
        \tfrac{1}{2}F^\top H^{-1} A + B^\top
    \right) \lambda  
                \right\}.
\end{split}
\tag{\ref{eq:opt_prob_dual}*}
\end{equation}
Furthermore, the following holds for the elements of the LM vector $\lambda = \left[ \lambda_1, \lambda_2,\cdots,\lambda_C\right]^\top$ {\color{black}\cite{boyd2004convex}}:
\begin{equation}\label{eq:LM_property}
    \lambda_i = 
    \left\{
        \begin{matrix}
            0 &  A_i^\top {\color{black}u^\star_{\textrm{unc}}} \leq B_i \\
            {\lambda}^\star_i & A_i^\top {\color{black}u^\star_{\textrm{unc}}} > B_i
        \end{matrix}
    \right.,\quad i\in\mathcal{C},
\end{equation}
where $ {\lambda}^\star_i > 0$, {\color{black}$u^\star_{\textrm{unc}} = - \tfrac{1}{2}H^{-1} F \in\mathbb{R}^m$ is the \textbf{unconstrained}} optimal solution to the PP and $A_i \in \mathbb{R}^m$ denotes the $i$-th column of $A$ while $B_i\in\mathbb{R}$ denotes the $i$-th element of $B\in\mathbb{R}^C$, for $i\in\mathcal{C}$.
{\color{black}
\begin{theorem}\label{thm:lp}
    The PP \eqref{eq:opt_prob_qp} is feasible iff the maximum value {  $d^\star = d({ \tilde{\lambda}})$ to the following LP:
    \begin{equation}\label{eq:dual_lp}
       { \tilde{\lambda}} \in \arg\underset{\lambda \in \Lambda }{\max}
                \left\{  - B^\top \lambda    
                \right\},
    \end{equation}
    is bounded, i.e., $d^\star < \infty$}, where in \eqref{eq:dual_lp} $\Lambda = \left\{ \left. x \in \textrm{null}(A) \subset \mathbb{R}^C \right| x \geq 0 \right\}$.
\end{theorem}
\begin{proof}
We begin the proof by decomposing the LM vector into two components:
    % \begin{equation}\label{eq:lambda_minimized_split}
      \(  \mathbb{R}^C\ni\lambda = \lambda_1 + \lambda_2 \),
    % \end{equation}
    where $\lambda_1 \in \textrm{col}(A), \lambda_2 \in \textrm{null}(A)$. Hence, \eqref{eq:opt_prob_dual*} becomes:
    \begin{equation}\label{eq:opt_prob_dual**}
    \begin{split}
        \left( \lambda_1^\star , \lambda^\star_2\right) = \underset{\lambda_1 + \lambda_2\geq 0 }{\textrm{arg}\max}
        \left\{
        T_1(\lambda_1) 
        +T_2(\lambda_2)
        \right\},
    \end{split}
    \tag{\ref{eq:opt_prob_dual*}*}
    \end{equation}
    where $T_1(\lambda_1) \triangleq -\tfrac{1}{4}\|\sqrt{H^{-1}}A\lambda_1 \|^2 - 
    \left(
        \tfrac{1}{2}F^\top H^{-1} A + B^\top
    \right) \lambda_1$ and $T_2(\lambda_2) \triangleq - B^\top \lambda_2$. The functions $T_1,T_2$ can be maximized independently as long as $\lambda_1 + \lambda_2\geq 0$. Finally, consider the set $\Lambda_+(\lambda_1) = \{ x \in \textrm{null}(A) : x \geq -\lambda_1\}$. 
    We distinguish two cases: 1) $\textrm{null}(A)\bigcap \{\lambda \geq 0\} = \{0\}$. Then $T_2$ in \eqref{eq:opt_prob_dual**} vanishes and owing to convexity of $T_1$ and $\{\lambda \geq 0\}$ \eqref{eq:opt_prob_dual**} admits a unique maximum $\lambda^\star = \lambda_1^\star$ thus \eqref{eq:opt_prob_qp} also admits a unique minimizer through \eqref{eq:stationary_kkt}. 2) $\textrm{null}(A)\bigcap \{\lambda > 0\}\neq \emptyset$.
    (\textit{Sufficiency:}) Assume that the maximum $d^\star$ of \eqref{eq:dual_lp} is bounded. Consider now the following problem, equivalent to \eqref{eq:opt_prob_dual**}:\(
    \underset{\lambda_2\geq -\lambda_1^\star,\lambda_2\in\textrm{null}(A) }{\max} \left\{ \underset{\lambda_1\in\textrm{col}(A)}{ \max } \left\{T_1(\lambda_1) + T_2(\lambda_2)\right\} \right\}\). Concerning the inner optimization problem w.r.t. $\lambda_1\in \textrm{col}(A)$, we have that $\lambda_1 = {N}_A s$, where ${N}_A \in \mathbb{R}^{C\times K_A}, s\in\mathbb{R}^{K_A}$, $K_A = C - \textrm{dim}(\textrm{ker}(A))$ such that $A N_A s \neq 0, \forall s\in\mathbb{R}^{K_A} $. Thus, $T_1(\lambda_1) = T_1( N_A s )$ is concave w.r.t. $s$, owing to its quadratic form and its maximization yields a single maximizer $s^\star :\lambda_1^\star = N_A s^\star$. Hence, \eqref{eq:opt_prob_dual**} becomes: \( \underset{\lambda_2\in\Lambda_+(\lambda_1^\star)}{\max} \left\{ T_2(\lambda_2) \right\}\), which we will show admits a bounded maximum, that corresponds to a set of maximizers $\lambda_2^\star$ such that $T_2(\lambda_2^\star)\leq 0$. First, we show boundedness. Assume that $T_2$ is unbounded on the set $\Lambda_+(\lambda_1^\star)$. It is easy to see that $\Lambda_+(\lambda_1^\star) \bigcap \Lambda \neq \emptyset$, hence unboundedness of $T_2$ on $\Lambda_+(\lambda_1^\star)$ implies its unboundedness on $\Lambda$, leading to a contradiction (we have assumed that \eqref{eq:dual_lp} is bounded). To show that $T_2(\lambda_2^\star)\leq 0$, consider the function $\beta(\theta) = (-B^\top \lambda_2^\star)\theta \in \textrm{null}(A), \theta\in\mathbb{R}$. Since $\lambda_2^\star \in \Lambda_+(\lambda_1^\star),\lambda_2^\star\neq \lambda_1^\star$, then $\exists \theta\in\mathbb{R}: \theta\lambda_2^\star \in \Lambda_+(\lambda_1^\star)$. Assuming for contradiction that $(-B^\top \lambda_2^\star)>0$, its derivative is $\tfrac{\textrm{d}\beta}{\textrm{d}\theta} = (-B^\top \lambda_2^\star)>0$. This implies that $\beta(\theta)$ can be increased, and thus $T_2(\theta\lambda_2^\star) = -B^\top (\theta \lambda_2^\star)$ can also be increased, leading to a contradiction as $\exists \theta\in\mathbb{R}: T_2(\theta\lambda_2^\star)>T_2(\lambda_2^\star)$ implying that $\lambda_2^\star$ is not a maximizer. Thus, \eqref{eq:opt_prob_dual*} attains a bounded maximum at the set of maximizers $\lambda^\star = \lambda_1^\star + \lambda_2^\star$ and the PP solution is $u^\star = -\tfrac{1}{2}H^{-1}\left( F + A\lambda^\star \right)= -\tfrac{1}{2}H^{-1}\left( F + A\lambda_1^\star \right)$, where $\lambda_2^\star$ is not necessarily unique. To show that this is indeed a feasible minimizer to \eqref{eq:opt_prob_qp}, the following holds from \eqref{eq:stationary_kkt}: ${d}(\lambda_1^\star + \lambda_2^\star) = \omega(u^\star) - B^\top\lambda_2^\star \Rightarrow \omega(u^\star) \geq {d}(\lambda^\star)$, since $(-B^\top \lambda_2^\star)<0$. Thus, the weak duality principle (Thm. 12.11 in \cite{Nocedal2006-dw}) is satisfied and $u^\star$ is indeed a feasible solution to \eqref{eq:opt_prob_qp}. (\textit{Necessity:}) Assume now that \eqref{eq:opt_prob_qp} is feasible, implying that a single maximizer $u^\star\in\mathbb{R}^m$ to \eqref{eq:opt_prob_qp} exists. Therefore, $\exists \lambda^\star = \lambda_1^\star + \lambda_2$ such that \eqref{eq:stationary_kkt} holds, for any $\lambda_2\in\textrm{null}(A)$. 
    Due to the uniqueness of $u^\star$ in \eqref{eq:stationary_kkt}, $\lambda_1^\star\in\textrm{col}(A)$ is unique. Thus, Eq. \eqref{eq:opt_prob_dual**} becomes: 
    % \begin{equation}
        \( \lambda_2^\star  \in  \underset{\lambda_2 \in \Lambda_+(\lambda_1^\star)}{\textrm{arg}\max}
        \left\{
        T_1(\lambda_1^\star) + T_2(\lambda_2) 
        \right\} = 
        \underset{\lambda_2 \in \Lambda_+(\lambda_1^\star)}{\textrm{arg}\max}
        \left\{
        T_2(\lambda_2) 
        \right\}\).
    %     \notag
    % \end{equation}
    By the weak duality principle (Thm. 12.11 in \cite{Nocedal2006-dw}), since \eqref{eq:opt_prob_qp} is by assumption feasible (thus its minimum is bounded), then \eqref{eq:opt_prob_dual*} is also bounded, rendering $\underset{\lambda \in \Lambda_+(\lambda_1^\star)}{\max}
        \left\{
        T_2(\lambda) 
        \right\}$ bounded as well. Assume for the sake of contradiction that the solution to \eqref{eq:dual_lp} is \textbf{\textit{unbounded}}, i.e., $T_2$ is unbounded on $\Lambda$. However, it is easy to see that $\Lambda_+(\lambda_1^\star) \bigcap \Lambda \neq \emptyset$. Therefore, if $T_2$ is unbounded on $\Lambda$, it is necessarily unbounded on $\Lambda_+(\lambda_1^\star)$, contradicting the weak duality principle. Therefore, the solution to \eqref{eq:dual_lp} is bounded, concluding the proof.
        
        % The intersection $\overline{\Lambda}(\lambda_1) \triangleq \Lambda\bigcap\Lambda_+(\lambda_1)$ is given by: $\overline{\Lambda}(\lambda_1) = \{ x \in\textrm{null}(A) | x > -\textrm{max}(0,\lambda_1)  \}$, where the operations are element-wise
    
    % Since the PP is feasible, {\color{red}it is easy to see that $\lambda_1^\star\geq 0$, otherwise the solution to the PP \eqref{eq:stationary_kkt} would violate the constraints.} Hence, $\Lambda \subseteq \Lambda_+(\lambda_1^\star)$, since $\lambda_1^\star\geq 0$. Assume now, for the sake of contradiction that \eqref{eq:dual_lp} is unbounded. Since it is unbounded on the set $\Lambda$ (by the above assumption) it is also unbounded on $\Lambda_+(\lambda_1), \forall \lambda_1\geq 0$. Hence, $T_2$ is unbounded. However this violates the weak duality principle (Thm. 12.11 in \cite{Nocedal2006-dw}), leading to a contradiction for the feasibility of the PP. Therefore, the solution to Problem \eqref{eq:dual_lp} can only be bounded, concluding the proof. 
\end{proof}
}
\subsection{ A Linear Program for evaluating feasibility}\label{sec:disreg_cons}
Consider the PP \eqref{eq:opt_prob_qp}, where the constraint matrices are split into hard and soft constraint ones, { i.e.: \( A = \left[ A_h, A_s\right]\in\mathbb{R}^{m \times C},\ B = \left[ B_h^\top, B_s^\top \right]^\top \in\mathbb{R}^{C}\) and \(A_h \in\mathbb{R}^{m \times C_h}, A_s \in\mathbb{R}^{m \times C_s}, B_h \in\mathbb{R}^{C_h},B_s\in\mathbb{R}^{C_s}\)}. The hard constraints must always be enforced, while the soft constraints can be disregarded if they compromise the feasibility of \eqref{eq:opt_prob_qp}. 
\begin{define}[\textbf{Disregarded Constraint}]
    { Consider Problem \eqref{eq:opt_prob_qp}. A constraint $\bar{A}^\top u \leq \bar{B}$, where $\bar{A}\in\mathbb{R}^m$ and $\bar{B}\in\mathbb{R}$, is defined as \textbf{disregarded} if its complementary set is enforced as a constraint in Problem \eqref{eq:opt_prob_qp} instead, i.e.:
    % \begin{equation}
    \(
        -\bar{A}^\top u \leq -\bar{B}. \)
    %     \notag 
    % \end{equation}
    }
\end{define}
In order to account for disregarded soft constraints, we employ the following notation. Given a set of soft constraints indexed by the set $\mathcal{C}_s\subset \mathcal{C}$ with cardinality $\mathbb{N} \ni C_s = \left|\mathcal{C}_s\right|$, consider the vector $P_i = \left[ p_1^i,p_2^i,\cdots,p_{C}^i\right]^\top\in\mathcal{P}\triangleq \{-1,+1\}^{C}, i\in\mathcal{C}_p$ with $\mathcal{C}_p = \{ 1,\cdots,C_p\}$, with $C_p = 2^{C}$, whose elements are given by:
% \begin{equation}
\(    p_j^i = 
    \left\{
        \begin{matrix}
            -1, & \textrm{if constraint $j$ is disregarded} \\
            +1, & \textrm{if constraint $j$ is \textbf{not} disregarded}
        \end{matrix}
    \right.,\)
    % \notag 
% \end{equation}
for some $j\in \{ 1,\cdots,C\}$ and $i\in\mathcal{C}_p$. This vector encodes any permutation of disregarded soft constraints for the original PP \eqref{eq:opt_prob_qp}.
\begin{define}\label{def:config}
    Given the PP \eqref{eq:opt_prob_qp}, the { vector} \(P_i\in\mathcal{P}, \textrm{ where }i\in\mathcal{C}_p,\) is called the \textbf{configuration} vector, or simply configuration, for \eqref{eq:opt_prob_qp}. Furthermore, the \textbf{level} of the configuration is defined as:
% \begin{equation}\label{eq:conf_lvl}
    \(\mathbb{N} \ni L(P_i) \triangleq  \underset{j\in\mathcal{C}}{\sum}\max\{p^i_j,0\}\). 
% \end{equation} 
\end{define}
The level counts the number of constraints { in a configuration} that are not disregarded. Finally, a configuration $P_i\in\mathcal{P},i\in\mathcal{C}_p$ is termed feasible if the following modified PP is feasible:
\begin{equation}\label{eq:opt_prob_qp*}
    \begin{split}
        &u^\star = \underset{u\in\mathbb{R}^m}{\textrm{arg}\min}\left\{\omega(u)\right\}, \\
        &\textrm{s.t.:}  \quad \ S(P_i)A^\top u \leq S(P_i)B,
    \end{split}
    \tag{\ref{eq:opt_prob_qp}*}    
\end{equation}
where:
% \begin{equation}\label{eq:sign_matrix}
%     \begin{split}
\(
        S(P_i) \triangleq \textrm{diag}
        \{
            P_i
        \} \in \mathbb{R}^{C \times C},
\)
%     \end{split}
% \end{equation}
and $\textrm{diag}(\cdot)$ denotes the square matrix with zero entries everywhere except for the diagonal, while its entries are the elements of the configuration vector $P_i$ for some $i\in\mathcal{C}_p$. The matrix $S(P_i)$ flips the signs of the disregarded soft constraints, compared to \eqref{eq:opt_prob_qp}. 
\begin{remark}
    The elements of the configuration vector that correspond to hard constraints are always set equal to $1$:
    % \begin{equation}
    \(
        p_j^i = 1,\ i\in\mathcal{C}_p,\ j\in\mathcal{C}_h.
    \)
    % \end{equation}
\end{remark}
To demonstrate how this formulation is useful, we prove the following lemma:
\begin{lemma}\label{lem:disregard_equiv}
    Consider the PP \eqref{eq:opt_prob_qp} consisting of the set of hard and soft constraints indexed by $\mathcal{C}_h\subset \mathcal{C}$ and $\mathcal{C}_s\subset \mathcal{C}$ respectively. Assume two configurations $P_{\textrm{feas}}, P_{\textrm{inf}}\in\{-1,+1\}^{C_s}$, where $P_{\textrm{feas}}$ is assumed to be feasible, while $P_{\textrm{inf}}$ is infeasible. Further assume w.l.o.g. that the two configurations differ\footnote{ Note that for the infeasible configuration to become feasible, some of the constraints need to be disregarded. Hence, for some $p_{id}^{inf}$ of the infeasible problem that correspond to non-disregrded constraints (and thus $p_{id}^{inf} = 1$), the associated constraints need to be disregarded, thus set to $p_{id}^{feas}=-1$.} over the index  set $\mathcal{C}_d \subseteq \mathcal{C}_s$ and $p^{\textrm{feas}}_{i_d} = -1$ and $p^{\textrm{inf}}_{i_d} = 1, \forall i_d\in\mathcal{C}_d$. The common indices are denoted as $\mathcal{C}_c = \mathcal{C}-\mathcal{C}_d$.
    Finally, consider two PPs: { the original PP \eqref{eq:opt_prob_qp} where the constraints indexed by $\mathcal{C}_d$ are dropped from the matrices $A,B$, and the} modified PP \eqref{eq:opt_prob_qp*}  for configuration $P_{\textrm{feas}}$. Then, these two PPs admit the same minimizer. 
\end{lemma}
\begin{proof}
    Consider the augmented cost function for Prob. \eqref{eq:opt_prob_qp}:
    \begin{equation}\label{eq:aug_fun_lem}
        d(u,\lambda) = 
        \omega(u)
                    + \underset{i\in\mathcal{C}_c}{\sum}\left( 
                        (A_s^i)^\top u -  B_s^i
                    \right)^\top\lambda_i,
    \end{equation}
    where each $\lambda_i\in\mathbb{R}_+,i\in\mathcal{C}_c$ obeys \eqref{eq:LM_property}. The optimal solution to this PP is denoted by the tuple $(u_1^\star,{ \tilde{\lambda}}_1)$, and the optimal augmented cost function is denoted by $d^\star_1 =  d(u_1^\star,{ \tilde{\lambda}}_1)$. {\color{black}Note that the minimizer $(u_1^\star,{ \tilde{\lambda}}_1)$ is unique, owing to convexity of the PP \eqref{eq:opt_prob_qp}}. Since the configuration $P_{\textrm{inf}}$ is infeasible, {then for the minimizer of \eqref{eq:aug_fun_lem} $(u_1^\star,{ \tilde{\lambda}}_1)$}:
    \begin{equation}\label{eq:constr_id_lem}
        -(A_s^{i_d})^\top u^\star_1 <  -B_s^{i_d}, \ \forall i_d\in\mathcal{C}_d.
    \end{equation}
    To see this, if there existed an optimal solution to \eqref{eq:opt_prob_qp} such that: $-(A_s^{i_d})^\top u^\star_1 \geq  -B_s^{i_d}\Leftrightarrow (A_s^{i_d})^\top u^\star_1 \leq  B_s^{i_d}$, $\forall i_d\in\overline{\mathcal{C}}_d$, for any $\overline{\mathcal{C}}_d \subseteq \mathcal{C}_d$, this would imply that the configuration $P_{\textrm{inf}}$ is not infeasible, leading to a contradiction. 
    Thus, constraint \eqref{eq:constr_id_lem} can be incorporated into the augmented function \eqref{eq:aug_fun_lem}. Since it holds, its associated LMs are $\lambda_{i_d} = 0,\ \forall i_d\in\overline{\mathcal{C}}_d$. Thus:
    \begin{equation}\label{eq:lem_1_d1}
    \begin{split}
        d_1^\star = d(u,\tilde{\lambda}_1)
                    % = \omega(u)
                    % &+ \underset{i\in\mathcal{C}_c}{\sum}\left( 
                    %     (A_s^i)^\top u -  B_s^i
                    % \right)^\top({ \tilde{\lambda}}_1)_i
                    - \underset{i_d\in\mathcal{C}_d}{\sum}\underbrace{\left( (A_s^{i_d})^\top u^\star  -B_s^{i_d}\right)^\top\lambda_{i_d}}_{=0},
    \end{split}
    \end{equation}
    where $({ \tilde{\lambda}}_1)_i\in\mathbb{R}_+$ denotes the $i$-th component of the vector ${ \tilde{\lambda}}_1$. However, note that the augmented cost function for Problem \eqref{eq:opt_prob_qp*} given the configuration $P_{\textrm{feas}}$ is identical to $d_1^\star$ \eqref{eq:lem_1_d1}. Thus, they share the same LM minimizer with the exception of $\#\left( \mathcal{C}_d \right)$ additional terms for Problem \eqref{eq:opt_prob_qp*}: ${ \tilde{\lambda}}_2 = \left[ ({ \tilde{\lambda}}_1)^\top, 0,\cdots,0\right]^\top$, where ${ \tilde{\lambda}_2\in\mathbb{R}^{C}}$ denotes the LM vector minimizer for Problem \eqref{eq:opt_prob_qp*}. 
\end{proof}
{ 
\begin{remark}
Lem. \ref{lem:disregard_equiv} demonstrates that switching the sign of the constraints (disregarding them) that render Problem \eqref{eq:opt_prob_qp} infeasible is equivalent to solving the same QP with the constraints completely omitted.
\end{remark}
}
% \begin{remark}\label{rem:lm_domain}
%     Instead of changing the sign of the disregarded constraints in the PP through the matrix $S(\cdot)$ \eqref{eq:opt_prob_qp*}, changing the domain of definition of the corresponding LMs in the DP achieves the same effect. The equivalent DP to the PP \eqref{eq:opt_prob_qp*}, given a configuration $P_i,i\in\mathcal{C}_p$ is:
%     \begin{equation}\label{eq:opt_prob_dual**}
%     \begin{split}
%     \lambda^\star &=\underset{\lambda\in\mathbb{R}^{C}}{\textrm{arg}\max}
%             \left\{
%                 -\tfrac{1}{4}\|\sqrt{H^{-1}}A\lambda \|^2 - 
%     \left(
%         \tfrac{1}{2}F^\top H^{-1} A + B^\top
%     \right) \lambda    
%             \right\} \\
%             &\textrm{s.t.:} \quad \lambda_j \in 
%             \left\{
%                 \begin{matrix}
%                     (-\infty, 0), & \textrm{if } p^i_j = -1 \\
%                     [0,+\infty), & \textrm{if } p^i_j = +1 
%                 \end{matrix}
%             \right. ,
%     \end{split}
% \tag{\ref{eq:opt_prob_dual*}*}
% \raisetag{20pt}
% \end{equation}
% where $\lambda_j\in\mathbb{R},j\in \mathcal{C}$ denotes the $j$-th element of the LM vector, and  $P_i = \left[ p_1^i,p_2^i,\cdots,p_{C}^i\right]^\top\in\mathcal{P}, i\in\mathcal{C}_p$.
% \end{remark}
% The above remark will be employed in the next subsection to formulate a method for evaluating the feasibility of any configuration. 
{ The following theorem is our main result and outlines the proposed LP for feasibility evaluation.}
%\subsection{Configuration Evaluation}\label{sec:conf_eval_graph}
\begin{theorem}
Problem \eqref{eq:opt_prob_qp*} under the configuration $P_i \in \mathcal{P}, i\in\mathcal{C}_p$ (See Def. \ref{def:config}) is feasible iff the following LP admits a bounded maximum, i.e, $d^\star < \infty$:
\begin{equation}\label{eq:dual_lp*}
\begin{split}
&\hspace{80pt}d^\star = \underset{\lambda \in \Lambda_i }{\max}
            \left\{  - B^\top \lambda    
            \right\},\\
&\textrm{where the bounds for the elements $\lambda_j\in\mathbb{R}$ of $\lambda\in\mathbb{R}^C$ are:}\\
&\Lambda_i = \left\{ \left. \lambda \in \textrm{null}(A) \subset \mathbb{R}^C \right| \lambda_j \in 
            \left\{
                \begin{matrix}
                    (-\infty, 0], & \textrm{if } p^i_j = -1 \\
                    [0,+\infty), & \textrm{if } p^i_j = +1 
                \end{matrix}
            \right. \right\},
\end{split}
\tag{\ref{eq:dual_lp}*}
\raisetag{60pt}
\end{equation}
% where 
% \begin{equation}
% \Lambda_i = \left\{ \left. \lambda \in \textrm{null}(A) \subset \mathbb{R}^C \right| \lambda_j \in 
%             \left\{
%                 \begin{matrix}
%                     (-\infty, 0), & \textrm{if } p^i_j = -1 \\
%                     [0,+\infty), & \textrm{if } p^i_j = +1 
%                 \end{matrix}
%             \right. \right\},\notag 
% \end{equation}            
where $P_i = \left[ p_1^i,p_2^i,\cdots,p_{C}^i\right]^\top\in\mathcal{P}, i\in\mathcal{C}_p$ and $\textrm{null}(\cdot)$ denotes the nullspace of a matrix.
\end{theorem}
{ 
\begin{proof}
According to Lem. \ref{lem:disregard_equiv}, disregarding a constraint is equivalent to negating the sign of a constraint. However, negating the sign of a constraint is equivalent to restricting the associated LM to lie on the interval $\lambda\in(-\infty,0]$. To see this, consider the problem with the negated constraints \eqref{eq:opt_prob_qp*}. According to Thm. \ref{thm:lp}, Prob. \eqref{eq:opt_prob_qp*} is feasible iff the solution to the LP:
\begin{equation}\label{eq:lp_thm2}
    d^\star = \underset{\lambda \in \Lambda }{\max}
                \left\{  - B^\top S(P_i) \lambda    
                \right\},
\end{equation}
is bounded, i.e., $d^\star < \infty$, where $\Lambda = \left\{ \left. \lambda \in \textrm{null}(A S(P_i)) \subset \mathbb{R}^C \right| \lambda \geq 0 \right\}$. Consider the transformation \( \mu = S(P_i)\lambda\). Then, \eqref{eq:lp_thm2} becomes
\begin{equation}\label{eq:lp_thm2_2}
    d^\star = \underset{\mu \in M }{\max}
                \left\{  - B^\top \mu    
                \right\},
\end{equation}
where \(\lambda \in \textrm{null}(AS(P_i)) \Rightarrow AS(P_i)\lambda = 0 \Rightarrow A\mu = 0 \Rightarrow \mu \in \textrm{null}(A)\). At the same time, $\lambda \geq 0$ implies that: \(  \mu_j \in 
            \left\{
                \begin{matrix}
                    (-\infty, 0], & \textrm{if } p^i_j = -1 \\
                    [0,+\infty), & \textrm{if } p^i_j = +1 
                \end{matrix}
            \right.\)
, which yields: 
\(
M_i = \left\{ \left. \mu \in \textrm{null}(A) \subset \mathbb{R}^C \right| \mu_j \in 
            \left\{
                \begin{matrix}
                    (-\infty, 0], & \textrm{if } p^i_j = -1 \\
                    [0,+\infty), & \textrm{if } p^i_j = +1 
                \end{matrix}
            \right. \right\}
\). However, problem \eqref{eq:lp_thm2_2} along with the definition of the set $M_i$ is exactly problem \eqref{eq:dual_lp*}, which concludes the proof.
\end{proof}
}
{ 
\begin{remark}
    The Problem Statement evidently can be reformulated only in terms of (non)emptiness of the feasible set $\{ u\in\mathbb{R}^m | A^\top u \leq B\}$, for which solutions exist  \cite{boyd2004convex,Nocedal2006-dw}. { Our analysis begins with the CBF-QP inspired PP \eqref{eq:opt_prob_qp}, which leads to the QP DP \eqref{eq:opt_prob_dual*}. The quadratic form of \eqref{eq:opt_prob_dual*} enables extracting condition \eqref{eq:dual_lp*} that assesses feasibility of \eqref{eq:opt_prob_qp}, implying non-emptiness of the feasible set. Our approach exhibits decreased execution time, which is supported theoretically (see the following remark) and validated through numerical simulations (see Section \ref{sec:meth_eval}).} 
    %However, the QP \eqref{eq:opt_prob_qp} is used to extract the DP QP \eqref{eq:opt_prob_dual*}. Most importantly, the condition which will be extracted in the sequel, only contains the inequality matrices $A,B$ and indeed computes (non)emptiness of the feasible set. { Since the feasible set (and the proposed condition) is independent of the QP's cost function, this evidently implies feasibility under any other cost function.}
\end{remark}
}
{ 
\begin{remark}\label{rem:boyd_lp}
 The proposed LP \eqref{eq:dual_lp*} consists of $C$ decision variables and is subject to $m$ linear \textbf{equality} constraints, and $C$ inequality constraints for the bounds of the LMs. In contrast, consider the LP proposed in \cite{boyd2004convex,Nocedal2006-dw}: 
 \begin{equation}\label{eq:lp_boyd}
\begin{split}
    &\underset{u\in\mathbb{R}^m,z\in\mathbb{R}^C}{\min} c^\top z\\
    &\mathrm{s.t.:}\quad A^\top u - z \leq B,\ z\geq 0,   
\end{split}
\end{equation}
which consists of $m+C$ decision variables and is subject to $2C$ linear \textbf{inequalities} and $C$ inequalities for the bounds of the auxiliary variable $z\in\mathbb{R}^C$. { Note that LP time complexity (see \cite{10.1145/3313276.3316303} for a recent result) is dependent on the number of variables/constraints. This theoretically indicates that the proposed method is expected to be solved faster} than the LP \eqref{eq:lp_boyd} { due to fewer decision variables and constraints}, which is supported by the numerical analysis that follows and further justifies the motivation and theoretical work that is the backbone of our approach.
\end{remark}
}
\section{ METHODOLOGY EVALUATION}\label{sec:meth_eval}
In the section we present two simulation case studies. The first study compares the computational load of our method \eqref{eq:dual_lp*} { against several different approaches.
%solving Problem \eqref{eq:opt_prob_qp}, { its dual, relevant LP-based methods \cite{boyd2004convex,Nocedal2006-dw} and Chinneck's method \cite{Chinneck1996}.} 
The second study illustrates an academic example inspired by CBF-QP controllers, demonstrating the applicability of our method in finding compatible constraints for online control synthesis.}
\subsection{Computational Load Comparison}\label{sec:computational-comparison}
{ 
Numerical evaluation to highlight the computational efficiency of the proposed method (i.e., Eq. \eqref{eq:dual_lp*}) is included in Figs. \ref{fig:comp_time_qp}, \ref{fig:comp_time_lp}. More specifically, Fig. \ref{fig:comp_time_qp} illustrates the feasibility evaluation time for the proposed method using the GLPK LP solver \url{https://www.gnu.org/software/glpk/} (last sub-figure) compared against the solution of the PP (first sub-figure), the DP (second sub-figure), and Chinneck's method \cite{Chinneck1996} with relaxation variables (third sub-figure) using MATLAB's QP solver. { This comparison demonstrates that the { pre-solution feasibility check} of { MATLAB's} QP solver is slower than \eqref{eq:dual_lp*}.} In Fig. \ref{fig:comp_time_lp} the proposed method is presented both for a MATLAB implementation and a GLPK one (second and last sub-figures) versus a ``Phase 1'' LP-based method \cite{boyd2004convex,Nocedal2006-dw} for feasibility evaluation, i.e. Eq. \eqref{eq:lp_boyd}.
\par 
The execution times are depicted through the colormap as a function of the number of constraints (varying up to 1000) and the number of dimensions of the PP's decision variables (varying up to 50). The relevant constraint matrices are generated randomly and thus the PP is probably infeasible as the number of constraints grows. Ten trials are executed for every problem instance for statistical significance, and the mean values are plotted.
The proposed method is significantly faster than the rest of the approaches. Since the PP is probably infeasible, the times in Fig. \ref{fig:comp_time_qp} reflect how fast the respective methods evaluate infeasibility, demonstrating the superiority of our method. Concerning Fig. \ref{fig:comp_time_lp}, the method described in Rem. \ref{rem:boyd_lp} is timed in two programming languages. Our approach is faster in both executions, corroborating the claim of Rem. \ref{rem:boyd_lp}.
}
\begin{figure*}
    \centering
    \includegraphics[trim={0cm 10cm 4.0cm 1cm},clip,width=1\textwidth]{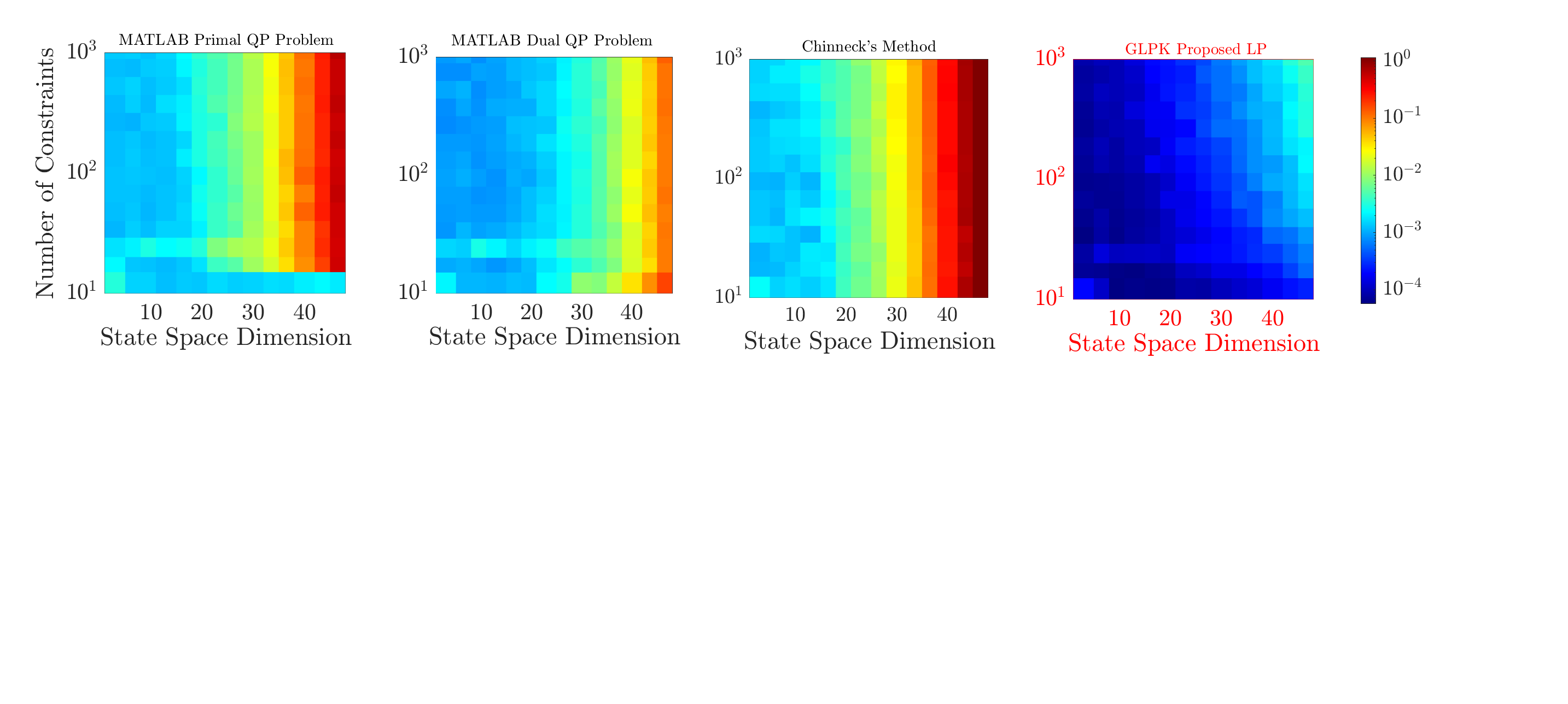}
    \caption{{  Several QP solvers compared to the proposed method. The PP (first sub-figure) the DP (second sub-figure), Chinneck's method \cite{Chinneck1996} (third sub-figure) and Proposed method \eqref{eq:dual_lp*} (fourth sub-figure).} The execution times are depicted through the colormaps in $[s]$.}
    \label{fig:comp_time_qp}
\end{figure*}
\begin{figure*}
    \centering
    \includegraphics[trim={0cm 10.5cm 3cm 1cm},clip,width=1\textwidth]{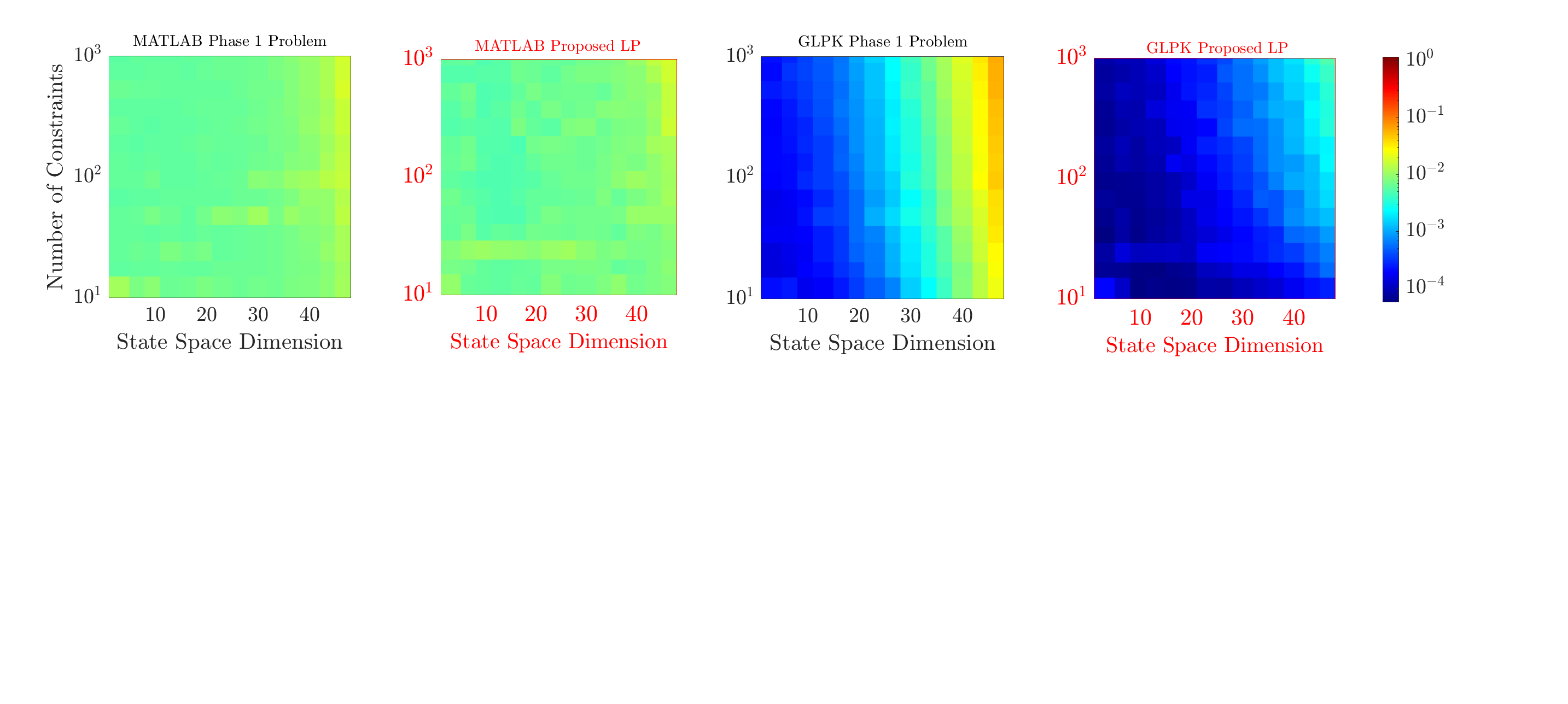}
    \caption{{  LP approaches compared to the proposed method. MATLAB implementations of \eqref{eq:lp_boyd} (first sub-figure) and the proposed one \eqref{eq:dual_lp*} (second sub-figure), GLPK implementations of \eqref{eq:lp_boyd} (third sub-figure) and the proposed one \eqref{eq:dual_lp*} (fourth sub-figure).} The execution times are depicted through the colormaps in $[s]$.}
    \label{fig:comp_time_lp}
\end{figure*}

\subsection{Constraint Selection Example}\label{sec:constr_sel}
{ 
The aim of this subsection is to demonstrate our feasibility evaluation method, i.e., Eq. \eqref{eq:dual_lp*}, in the context of control synthesis, and specifically for CBF-QP controllers (see Rem. \ref{rem:results}). The constraints as well as the matrices of problem \eqref{eq:opt_prob_qp} are evolving in time; however, for simplicity, we drop the time argument, i.e., we denote $H:\mathbb{R}_+{ \rightarrow}\mathbb{R}^{m\times m}$ a symmetric positive definite matrix, $F:\mathbb{R}_+{ \rightarrow}\mathbb{R}^m$, $A:\mathbb{R}_+{ \rightarrow}\mathbb{R}^{m\times C},B:\mathbb{R}_+{ \rightarrow}\mathbb{R}^{C}$. The decision variable corresponds to a two-dimensional input vector and the hard constraints correspond to time-varying input bounds (see also the left subfigure of Fig. \ref{fig:example}, where a snapshot of the QP is depicted). Furthermore, five soft constraints correspond to CBF/CLF conditions that can be disregarded in favor of satisfying the input bounds. The feasibility evaluation method \eqref{eq:dual_lp*} is implemented at discrete time instances to choose a feasible configuration.}
% The aim of this subsection is to demonstrate our method in the context of control synthesis.
% As stated in Rem. \ref{rem:results}, this example is inspired by CBF-QP controllers. The decision variables correspond to the elements of a two-dimensional input vector, and the hard constraints correspond to time-varying input bounds (see Fig. \ref{fig:example}, where the QP is depicted for a \textbf{specific} time instance). Furthermore, five soft constraints correspond to CBF/CLF conditions that are amenable to being disregarded in favor of satisfying the input bounds. The constraints as well as the matrices of problem \eqref{eq:opt_prob_qp} are evolving in time, i.e., $H:\mathbb{R}_+{ \rightarrow}\mathbb{R}^{m\times m}$ is a symmetric positive definite matrix, $F:\mathbb{R}_+{ \rightarrow}\mathbb{R}^m$, $A:\mathbb{R}_+{ \rightarrow}\mathbb{R}^{m\times C},B:\mathbb{R}_+{ \rightarrow}\mathbb{R}^{C}$. Due to this time dependence, besides evaluating the feasibility of { any} constraint configuration, a planning algorithm is needed to choose among the { feasible configurations}. 
\par
In this example we solve the maxFS problem \cite{doi:10.1080/03155986.2019.1607715}, i.e., maximize the level of the configuration { at discrete time instances}. We thus begin by introducing a structure onto the space of permutations $\mathcal{P}$ defined in subsection \ref{sec:disreg_cons}. Let $\mathcal{G} = \left\{\mathcal{V},\mathcal{E}\right\}$ denote a graph, where $\mathcal{V} = \left\{ P_i \;|\; i\in\mathcal{C}_p\right\}$ denotes the vertex set and \(\mathcal{E}\) denotes the edge set with \((i,j) \in \mathcal{E}\) iff:
% \begin{equation}
\(
    \left\| 
        P_i - P_j
    \right\|_1 = 2, \ i\neq j,\ i,j\in\mathcal{C}_P,
\)
% \end{equation}
where $\|\cdot\|_1$ denotes the 1-norm. In other words, two vertices of $\mathcal{G}$ are connected if the corresponding configurations differ at a single element, i.e., a single constraint changes between disregarded/not disregarded { between two neighboring configurations $P_i,P_j$}. 
%This graph structure can be employed to inform on methods for online constraint selection. 
The configuration graph for this example at a specific time instance is also presented in the right subfigure of Fig. \ref{fig:example}, { where the links between two nodes indicate that the corresponding configurations differ at a single constraint}.
% \par 
{ In order to choose among the feasible configurations, we { employ} two algorithms, one exhaustive and one heuristic. These algorithms are not novel, i.e., are not part of the contribution of the paper, but rather are used to demonstrate the applicability of the method to feasible constraint selection}. 
\subsubsection{Greedy Approach}
A greedy approach consists of evaluating the feasibility of all possible configurations, and choosing the constraint with the maximal level. It is implemented at each time instance to choose an appropriate set { of constraints to be disregarded} in terms of Problem \eqref{eq:opt_prob_qp}. Since all possible configurations are evaluated, this algorithm scales { according} to $\mathcal{O}( 2^{C} N(C))$, where $N(C)$ is the execution time for evaluating a configuration of $C$ constraints and $C\in\mathbb{N}$ is the number of constraints. 
% \begin{algorithm}[h]
% \caption{GREEDY ALGORITHM}
% \begin{algorithmic}
% \STATE $\bullet$ Given the PP \eqref{eq:opt_prob_qp} and its associated DP \eqref{eq:opt_prob_dual**},
% \FOR{$i\in\mathcal{C}_p$}
%     \STATE $\bullet$ Solve the LP \eqref{eq:dual_lp*} for the configuration $P_i \in \mathcal{P}$,
%     \STATE $\bullet$ Compute the level of the permutation $P_i$: $L_i \triangleq L(P_i) =  \underset{j\in\mathcal{J}}{\sum}p^i_j$.
% \ENDFOR
% \STATE $\bullet$ Return the configuration of maximal level:
% % \begin{equation}
% \(
%     P^\star = \underset{ P_i | i\in\mathcal{C}_p }{\textrm{arg}\max }\left\{ L(P_i) \right\}.\notag 
% \)
% % \end{equation}
% \end{algorithmic}
% % %\vspace{-0.5cm}
% \label{alg:greedy}
% \end{algorithm}

\subsubsection{Heuristic Search}
The heuristic method restricts the search to neighbors of the current configuration of $\mathcal{G}$. Since each node is connected to nodes whose permutation vector differs at a single digit, this yields $C$ neighbors in total. The algorithm only evaluates the neighbors of the current vertex of the graph, and therefore it scales { according} to $\mathcal{O}(N(C) C )$.
% , where $N(C)$ is the execution time for evaluating a configuration and $C\in\mathbb{N}$ is the number of constraints. 
However, it does not necessarily provide the optimal configuration. 
% \begin{algorithm}[h]
% \caption{HEURISTIC SEARCH ALGORITHM}
% \begin{algorithmic}
% \STATE $\bullet$ Given the PP \eqref{eq:opt_prob_qp} and its associated DP \eqref{eq:opt_prob_dual**}, as well as a configuration $P_I\in\mathcal{P},I\in\mathcal{C}_p$,
% \STATE $\bullet$ Find the neighbor set $\mathcal{N}_I \subset \mathcal{V}$ of the configuration $P_I$,
% \FOR{$P_i\in\mathcal{N}_I$}
%     \STATE $\bullet$ Solve the LP \eqref{eq:dual_lp*} for the configuration $P_i \in \mathcal{P}$,
%     \STATE $\bullet$ Compute the level of the permutation $P_i$: $L_i \triangleq L(P_i) =  \underset{j\in\mathcal{J}}{\sum}p^i_j$.
% \ENDFOR
% \STATE $\bullet$ Return the configuration of maximal level:
% % \begin{equation}
% \(
%     P^\star = \underset{ P_i | i\in\mathcal{N}_I \cup P_I }{\textrm{arg}\max }\left\{ L(P_i) \right\}.\notag 
% \)
% % \end{equation}
% \end{algorithmic}
% % %\vspace{-0.5cm}
% \label{alg:heuristic}
% \end{algorithm}

% \begin{remark}
% The herein proposed algorithms only serve to demonstrate the applicability of our method { in the context of feasible constraint configuration selection. For instance, a more effective approach might consist of formulating an optimization problem based on a dynamic version of the aforementioned graph $\mathcal{G}$ and constructing more complex search-based methods}. We leave this as part of future research, out of the scope of this paper.  
% \end{remark}

These two search methods are employed to choose the configuration of feasible constraints and the results are presented in { Fig. \ref{fig:example}} and the video: \url{https://vimeo.com/1021254998?share=copy#t=0}, { which reveal that the method accurately labels the configurations as feasible/infeasible.} { More specifically in Fig. \ref{fig:example}, soft constraints 3,4 are not compatible with the hard constraints and the relevant configurations are thus infeasible (see graph in right figure). Any other configuration is feasible since constraints 1, 2 and 5 are compatible with the hard constraints. The maximal level configuration is selected, disregarding constraints 3 and 4.}

% In this example, the greedy approach is tractable due to the low number of constraints, and the local heuristic is able to select the configuration that maximizes the number of active constraints. We underline that the aforementioned tractability, as well as the ability of the heuristic method to find the maxFS solution, do not hold in general, especially in case of larger problems, which highlights the necessity for more efficient algorithms for configuration selection. 
\begin{figure*}
    \centering
    \includegraphics[trim={0cm 14cm 27cm 0cm},clip,width=0.62\textwidth]{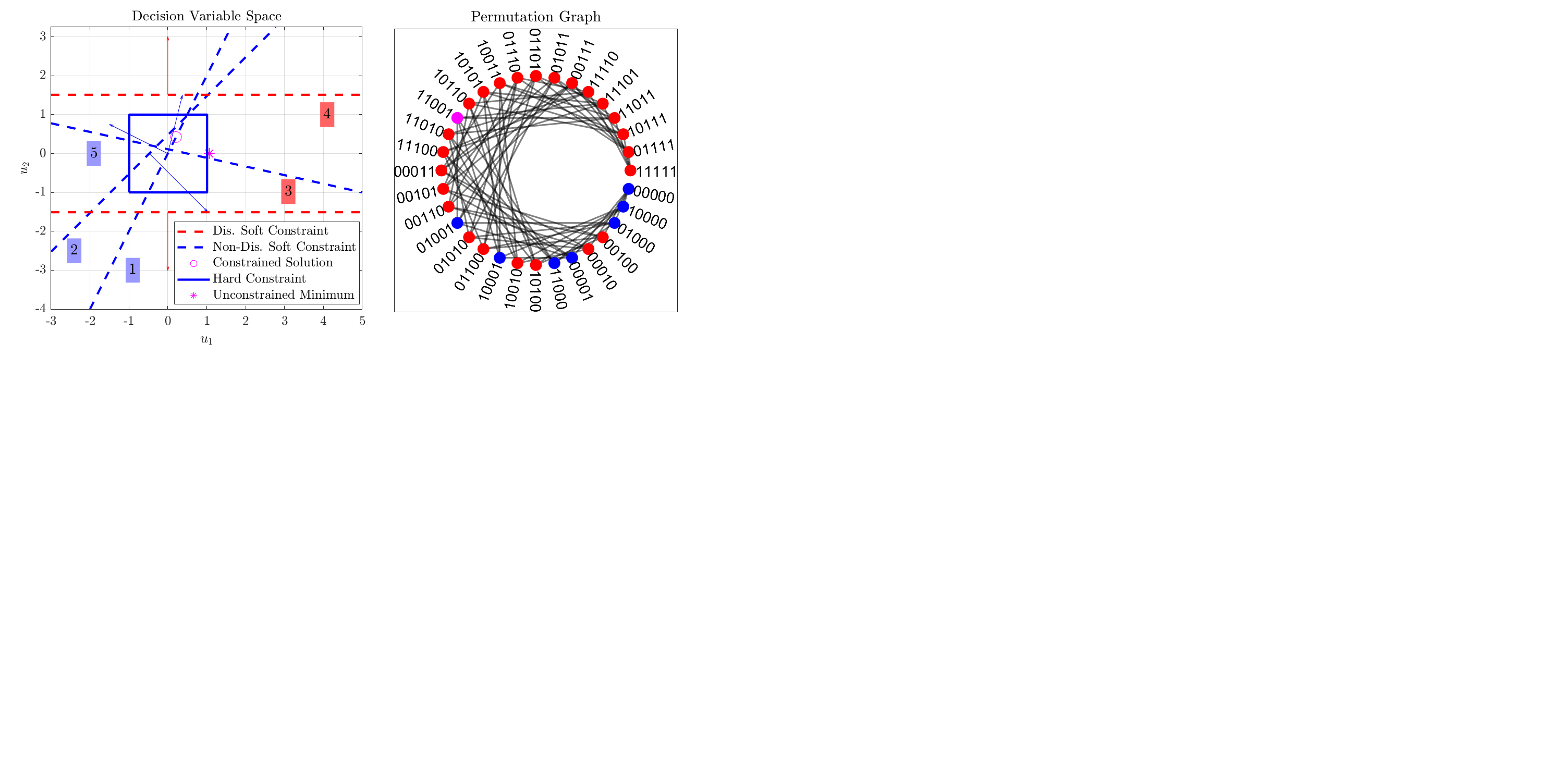}
    \caption{An example of a constrained QP problem over the decision variable space with hard constraints (solid lines), { not-disregarded} soft constraints (dashed blue lines) and { disregarded} soft constraints (red dashed lines) is depicted in the left figure. The arrows point to the half-plane of the respective constraint. The unconstrained optimal solution (magenta star) and the constrained optimal solution (magenta circle) are also depicted. { Each soft constraint is labeled through numbers from one to five.} The corresponding configuration graph is depicted in the right figure, with the feasible configurations highlighted in blue, the infeasible configurations highlighted in red and the selected configuration highlighted in magenta. For readability, the configurations are depicted as binary numbers, where ``0'' in the $i$-th digit denotes that the $i$-th constraint is disregarded and ``1'' is employed otherwise.}
    \label{fig:example}
\end{figure*}

% \begin{figure}
%     \centering
%     \includegraphics[trim={5.0cm 0cm 5.25cm 0.cm},clip,width=0.35\textwidth]{FIGURES/plot_times_2.eps}
%     \caption{The execution times for Alg. \ref{alg:greedy} (blue dots) and Alg. \ref{alg:heuristic} (red dots), as a function of the run-time of the method. }
%     \label{fig:example_times}
% \end{figure}

\section{CONCLUSION}\label{sec:conclusions}
In this paper, a method for evaluating the feasibility of any configuration of soft constraints in QP problems is presented, using duality principle and a properly defined LP that can be solved more efficiently than LPs in similar existing methodologies. %This method can be employed for QP-related controllers and aids in a-priori selecting a feasible constraint configuration without solving the QP problem itself. 
The proposed approach can serve as a basis for assessing the feasibility of online constrained-control techniques such as MPC and CBF-QPs. In our future work, we aim to develop planners to recursively maintain or minimally-relax feasibility based on our approach, { along with more complex optimization criteria than the cardinality of the constraint set, e.g., accounting for constraint hierarchy}.

%%%%%%%%%%%%%%%%%%%%%%%%%%%%%%%%%%%%%%%%%%%%%%%%%%%%%%%%%%%%%%%%%%%%%%%%%%%%%%%%

\bibliographystyle{IEEEtran}
\bibliography{bib}

\end{document}